\documentclass[final,times,authoryear,1p]{elsarticle}

\usepackage{amsmath,amsfonts,amssymb}
\usepackage{amsthm}
\usepackage{graphicx}
\usepackage{color}

\usepackage{fancyhdr}
\newtheorem{theorem}{Theorem}[section]
\newtheorem{lemma}[theorem]{Lemma}
\newtheorem{proposition}[theorem]{Proposition}

\newtheorem{definition}[theorem]{Definition}
\newtheorem{remark}[theorem]{Remark}

\usepackage{amssymb}

\journal{Transportation Research Part B}

\begin{document}

\begin{frontmatter}



 \begin{center}
\textcolor{blue}{ARTICLE LINK:  http://www.sciencedirect.com/science/article/pii/S0191261513000209
\\  PLEASE CITE THIS ARTICLE AS\\ 
Han, K., Friesz, T.L., Yao, T., 2013. Existence of simultaneous route and departure choice dynamic user equilibrium. Transportation Research Part B 53, 17-30.}
 \line(1,0){469}
 \end{center}

\title{Existence of Simultaneous Route and Departure Choice Dynamic User Equilibrium\tnoteref{t1}}

\tnotetext[t1]{This work is partially supported by NSF through grant EFRI-1024707, ``A theory of complex transportation network design".}

\author[math]{Ke Han}
\ead{kxh323@psu.edu}

\author[ie]{Terry L. Friesz\corref{cor}}
\ead{tfriesz@psu.edu}

\author[ie]{Tao Yao}
\ead{tyy1@engr.psu.edu}

\cortext[cor]{Corresponding author}

\address[math]{Department of Mathematics, Pennsylvania State University, PA 16802, USA.}
\address[ie]{Department of Industrial and Manufacturing Engineering, Pennsylvania State University, PA 16802, USA.}

\begin{abstract}
This paper is concerned with the existence of the \textit{simultaneous
route-and-departure choice dynamic user equilibrium} (SRDC-DUE) in continuous time.
The SRDC-DUE problem was formulated as an infinite-dimensional variational
inequality in \cite{Friesz1993}. In deriving our existence result, we employ
the \textit{generalized Vickrey model} (GVM)  introduced in \cite{HFY1, HFY2}
 to formulate the underlying network loading problem. As we explain, the GVM
corresponds to a path delay operator that is provably strongly continuous on
the Hilbert space of interest. Finally, we provide the desired SRDC-DUE
existence result for general constraints relating path flows to a table of
fixed trip volumes without invocation of {\it a priori} bounds on the
path flows.
\end{abstract}


\end{frontmatter}

\section{\label{Intro}Introduction}

In this paper we shall consider \textit{dynamic traffic assignment} (DTA) to
be the positive (descriptive) modeling of time-varying flows of automobiles
on road networks consistent with established traffic flow theory and travel
demand theory. \textit{Dynamic User Equilibrium} (DUE) is one type of DTA
wherein effective unit travel delay for the same purpose is identical for
all utilized path and departure time pairs. The relevant notion of travel
delay is effective unit travel delay, which is the sum of arrival penalties
and actual travel time. For our purposes in this paper, DUE is modeled for
the within-day time scale based on fixed travel demands.

In the last two decades there have been many efforts to develop a
theoretically sound formulation of dynamic network user equilibrium that is
also a canonical form acceptable to scholars and practitioners alike. DUE
models tend to be comprised of four essential sub-models:

\begin{enumerate}
\item a model of path delay;

\item flow dynamics;

\item flow propagation constraints;

\item a path/departure-time choice model.
\end{enumerate}

\noindent Furthermore, analytical DUE models tend to be of two varieties:
(1) \textit{route choice} (RC) user equilibrium \citep{Friesz1989, MNa, MNb, Mounce2006, SW2, ZM}; and (2) \textit{simultaneous route-and-departure} \textit{choice} (SRDC) dynamic user equilibrium \citep{Friesz1993, FBST, FKKR, LWRDUE, RHB, WTC}. For both types of DUE models, the existence of a dynamic user equilibrium in continuous time remains a
fundamental issue. A proof of DUE existence is a necessary foundation for
qualitative analysis and computational studies. In this paper, we provide a
DUE existence result for the SRDC DUE problem when it is formulated as an
infinite-dimensional variational inequality of the type presented in \cite{Friesz1993}. In this paper, in order to establish a DUE existence result,
we study the network loading problem based on the generalized Vickrey model (GVM) proposed in \cite{HFY1, HFY2}. All of our results
presented in this paper are more general than any obtained previously for
DUE when some version of the point queue model is employed.

\subsection{Formulation of the SRDC user equilibrium}

\label{introformulation}

There are two essential components within the RC or SRDC notions of DUE: (i)
the mathematical expression of Nash-like equilibrium conditions, and (ii) a
network performance model, which is, in effect, an embedded network loading
problem. The embedded network loading problem captures the relationships
among arc entry flow, arc exit flow, arc delay and path delay for any path
departure rate trajectory. 


There are multiple means of expressing the Nash-like notion of a dynamic
equilibrium, including the following:

\begin{enumerate}
\item a variational inequality \citep{Friesz1993, SW1, SW2}

\item an evolution equation in an appropriate
function space \citep{Mounce2006, SW2}

\item a nonlinear complementarity problem \citep{WTC, HUD}

\item a differential variational inequality \citep{FBST, FKKR, LWRDUE, Friesz and Mookherjee}; and

\item a differential complementarity system \citep{Pang}.
\end{enumerate}

\noindent The variational inequality representation is presently the primary
mathematical form employed for both RC and SRDC DUE. The most obvious
approach to establishing existence for any of the mathematical
representations mentioned above is to convert the problem to an equivalent
fixed point problem and then apply Brouwer's fixed point existence theorem.
Alternatively, one may use an existence theorem for the particular
mathematical representation selected; it should be noted that most such
theorems are derived by using Brouwer's famous theorem. So, in effect, all
proofs of DUE existence employ Brouwer's fixed point theorem, either
implicitly or explicitly. One statement of Brouwer's theorem appears as
Theorem 2 of \cite{Browder}. Approaches based on Brouwer's theorem require
the set of feasible path flows (departure rates) under consideration to be
compact and convex in a Banach space, and typically involve an {\it a
priori} bound on all the path flows.

We also wish to point out that this paper employs much more general
constraints relating path flows to a table of fixed trip volumes than has
been previously considered when studying SRDC-DUE. Moreover, in our study of
existence, we do not invoke {\it a priori} bounds on the path flows to
assure boundedness needed for application of Brouwer's theorem. That is, a
goal of this paper is to investigate the existence of DUE without making the
assumption of \textit{a priori} bounds for departure rates. Note should be
taken of the following fact: the boundedness assumption is less of an issue
for the RC DUE by virtue of problem formulation; that is, for RC\ DUE, the
travel demand constraints are of the following form: 
\begin{equation}
\sum_{p\in \mathcal{P}_{ij}}h_{p}(t)~=R_{ij}(t)\qquad \forall ~t,\quad
\forall ~(i,\,j)\in \mathcal{W}  \label{introeqn1}
\end{equation}%
where $\mathcal{W}$ is the set of origin-destination pairs, $\mathcal{P}_{ij}
$ is the set of paths connecting $\left( i,j\right) \in \mathcal{W}$ and $
h_{p}(t)$ is the departure rate along path $p$. Furthermore, $R_{ij}(t)$
represents the rate (not volume) at which travelers leave origin $i$ with
the intent of reaching destination $j$ at time $t$; each such trip rate is
assumed to be bounded from above. Since (\ref{introeqn1}) is imposed
pointwise and every path flow $h_{p}$ is nonnegative, we are assured that
each $h=\left( h_{p}:~p\in \mathcal{P}_{ij},~\left(i,\,j\right) \in \mathcal{W}\right) $ are automatically uniformly bounded. On the other hand, the
SRDC user equilibrium imposes the following constraints on path flows: 
\begin{equation}
\sum_{p\in \mathcal{P}_{ij}}\int_{t_{0}}^{t_{f}}h_{p}(t)\,dt~=~Q_{ij}\qquad
\forall ~(i,\,j)\in \mathcal{W}  \label{introeqn2}
\end{equation}%
where $Q_{ij}\in \mathcal{\Re }_{+}^{1}$ is the volume (not rate) of
travelers departing node $i$ with the intent of reaching node $j$. The
integrals in (\ref{introeqn2}) are interpreted as Legesgue; hence, (\ref%
{introeqn2}) alone is not enough to assure bounded path flows. This
observation has been the major hurdle to proving existence without the 
{\it a priori} invocation of bounds on path flows. In this paper, we will
overcome this difficulty through careful analysis of the GVM and by
investigating the effect of user behavior in shaping network flows, in a
mathematically intuitive yet rigorous way.

\subsection{Importance of the path delay operator}

Clearly another key component of continuous-time DUE is the path delay
operator, typically obtained from \textit{dynamic network loading} (DNL),
which is a subproblem of a complete DUE model\footnote{%
Note that, by referring to the network loading procedure, we are neither
employing nor suggesting a sequential approach to the study and computation
of DUE. Rather a subset of the equations and inequalites comprising a
complete DUE model may be grouped in a way that identifies a traffic
assignment subproblem and a network loading subproblem. Such a grouping and
choice of names is merely a matter of convenient language that avoids
repetitive reference to the same mathematical expressions. Use of such
language does not alter the need to solve both the assignment and loading
problems consistently and, thus, simultaneously. A careful reading of the
mathematical presentation made in subsequent sections makes this quite clear.%
}. Any DNL must be consistent with the established path flows and link delay
model, and DNL is usually performed under the \textit{first-in-first-out}
(FIFO) rule. The properties of the delay operator are critical to proving
existence of a solution to the infinite-dimensional variational inequality
used to express DUE. In \cite{ZM}, using the \textit{link delay model}
introduced by \cite{Friesz1993}, the authors showed weak continuity of the
path delay operator under the assumption that the path flows are \textit{a
priori} bounded. Their continuity result is superceded  by a
more general result proven in this paper: the path delay operator of
interest is strongly continuous without the assumption of boundedness.
Strong continuity without boundedness is central to our proof of existence
in the present paper. 

In this paper, as a foundation for DNL, we will consider the Vickrey model of
congestion first introduced by \cite{Vickrey} and later studied by  \cite{HFY1, HFY2}. The Vickrey model for a single link is primarily described by an 
\textit{ordinary differential equation} (ODE) with discontinuous right hand
side. Such irregularity has made it difficult to analyze the Vickery model in
continuous time. Fortunately, in this paper, we will be able to take
advantage of the closed-form reformulation proposed in \cite{HFY1, HFY2}, then prove the strong
continuity of the path delay operator without boundedness of the path flows. This will provide  a quite general existence proof for SRDC-DUE based on the generalized Vickrey model.

\subsection{Organization}

The balance of this paper is organized as follows. Section \ref{Prel}
provides essential mathematical background on the concepts that will be used
in the paper. Section \ref{DUE} briefly reviews the formal definition of
dynamic user equilibrium and its formulation as a variational inequality.
Section \ref{DNL} recaps the generalized Vickrey model (GVM) originally
put forward by \cite{HFY1, HFY2}. Section \ref{dueexistence} formally
discusses the properties of the effective delay operator. The main result of
this paper, the existence of an SRDC-DUE when the GVM informs network
loading is established in Theorem \ref{gvmthm} of Section \ref{visec}.

\section{\label{Prel} Mathematical preliminaries}

A \textit{topological vector space} is one of the basic structures
investigated in functional analysis. Such a space blends a topological
structure with the algebraic concept of a vector space. The following is a
precise definition.

\begin{definition}
\label{tvdef} \textbf{(Topological vector space)} A topological vector space 
$X$ is a vector space over a topological field $\mathbb{F}$ (usually the
field of real or complex numbers with their standard topologies) which is endowed
with a topology such that vector addition $X\times X\rightarrow X$ and scalar
multiplication $\mathbb{F}\times X\rightarrow X$ are continuous functions. 
\end{definition}

\noindent As a consequence of Definition \ref{tvdef}, all normed vector
spaces, and therefore all Banach spaces and Hilbert spaces, are examples of
topological vector spaces. Also important is the notion of a seminorm:

\begin{definition}
\label{seminormdef} \textbf{(Seminorm)} A seminorm on a vector space $X$ is
a real-valued function $p$ on $X$ such that

\begin{itemize}
\item[(a)] $p(x+y)~\leq~p(x)+p(y)$

\item[(b)] $p(\alpha\,x)~=~|\alpha|\,p(x)$
\end{itemize}

for all $x$ and $y$ in $X$ and all scalars $\alpha$.
\end{definition}

\begin{definition}
\label{lctvdef} \textbf{(Locally convex space)} A locally convex space is
defined to be a vector space $X$ along with a family of seminorms $%
\{p_i\}_{i\in\mathcal{I}}$ on $X$.
\end{definition}

\noindent As part of our review we make note of the following essential
knowledge:

\begin{flushleft}
\textbf{Fact 1.} The space of square-integrable real-valued functions on a
compact interval $[a,\,b]$, denoted by $\mathcal{L}^{2}[a,\,b]$, is a
locally convex topological vector space.

\medskip \noindent \textbf{Fact 2.} The $m$-fold product of the spaces of
square-integrable functions $\big(\mathcal{L}^{2}[a,\,b]\big)^{m}$ is a locally
convex topological vector space.
\end{flushleft}

\begin{definition}
\label{dualdef} \textbf{(Dual space)} The dual space $X^{\ast }$ of a vector
space $X$ is the space of all continuous linear functions on $X$.
\end{definition}

\noindent Given a vector space $X$, let $\varphi\in X^*$ be a continuous linear function on $X$, then we use $\left<\cdot,\,\cdot\right>$ to denote the duality pairing of $X$ with its dual space $X^*$, that is
$$
\left<\varphi,\,x\right>~\doteq~\varphi(x)\qquad\forall x\in X
$$

\noindent Another key property we consider without proof is:

\begin{flushleft}
\textbf{Fact 3.} The dual space of $\mathcal{L}^{p}[a,\,b]$ for $1<p<\infty $ has a
natural isomorphism with $\mathcal{L}^{q}[a,\,b]$ where $q$ is such that $1/p+1/q=1$.
In particular, the dual space of $\mathcal{L}^{2}[a,\,b]$ is again $%
\mathcal{L}^{2}[a,\,b]$.
\end{flushleft}

\noindent Let us now give the formal definition of a variational inequality
in a topological setting:

\begin{definition}
\textbf{(Infinite-Dimensional Variational inequality)} Let $V$ be a
topological vector space and $F:\,U \rightarrow V^*$, where $U\subset V$. The infinite-dimensional variational inequality is
posed as the following problem 
\begin{equation}
\left. 
\begin{array}{c}
\hbox{find}~u^{\ast}\in U ~~\hbox{such that} \\ 
\left\langle F(u^*),\,u-u^{\ast }\right\rangle ~\geq ~0~~\forall ~u\in U%
\end{array}
\right\} ~VI(F,\,U)  \label{videf}
\end{equation}
\end{definition}

\noindent The key foundation for analysis of existence is the following
theorem given in \cite{Browder}:

\begin{theorem}
\label{mainthm} Let $K$ be a compact convex subset of the locally convex
topological vector space $E$, $T$ a continuous (single-valued) mapping of $K$
into $E^{\ast }$. Then there exits $u_{0}$ in $K$ such that 
\begin{equation*}
\Big<T(u_{0}),\,u-u_0\Big>~\geq ~0
\end{equation*}%
for all $u\in K$.
\end{theorem}
\begin{proof}
See \cite{Browder}.
\end{proof}

\begin{definition}
\textbf{(Compactness of subspaces)} A subset $K$ of a topological space $X$
is called compact if for every arbitrary collection $\left\{ U_{\alpha
}\right\} _{\alpha \in A}$ of open subsets of $X$ such that 
\begin{equation*}
K~\subset ~\bigcup_{\alpha \in A}U_{\alpha }
\end{equation*}
where $A$ is an arbitrary index set, there is a finite subset $I$ of $A$ such that 
\begin{equation*}
K~\subset ~\bigcup_{i\in I}U_{i}
\end{equation*}
\end{definition}

\begin{definition}
\textbf{(Sequential compactness)} A topological space is sequentially
compact if every sequence has a convergent subsequence.
\end{definition}

\noindent An outgrowth of the concepts and results given above, the
following fact is stated without proof:

\begin{flushleft}
\textbf{Fact 4.} \citep{Royden} In metric space (hence topological vector space), the
notions of compactness and sequential compactness are equivalent.
\end{flushleft}

\noindent The final bit of specialized knowledge about topological vector
spaces that we shall need is the following:

\begin{definition}
\textbf{(Weak convergence in Hilbert space)} A sequence of points $\left\{
x_{n}\right\} $ in a Hilbert space $\mathcal{H}$ is said to be convergent
weakly to a point $x\in \mathcal{H}$, denoted as $x_{n}\rightharpoonup x$ if 
\begin{equation*}
\left\langle x_{n},\,y\right\rangle ~\rightarrow ~\left<
x,\,y\right>  \qquad n~\rightarrow~\infty
\end{equation*}%
for all $y\in \mathcal{H}$, where $\left\langle \cdot ,\,\cdot \right\rangle 
$ is the inner product on the Hilbert space.
\end{definition}

\section{\label{DUE} Continuous-time dynamic user equilibrium}

 In this section, we will assume the time interval of interest is 
\begin{equation*}
\lbrack t_{0},\,t_{f}]\subset \mathcal{\Re }^{1}
\end{equation*}%
The most crucial component of the DUE model is the path delay operator,
which provides the time to traverse any path $p$ per unit of flow departing
from the origin of that path. The delay operator is denoted by 
\begin{equation*}
D_{p}(t,\,h)\qquad \forall ~p\in \mathcal{P}
\end{equation*}%
where $\mathcal{P}$ is the set of all paths employed by network users, $t$
denotes the departure time, and $h$ is a vector of departure rates.
Throughout the rest of the paper, we stipulate that 
\begin{equation*}
h\in \Big(\mathcal{L}_{+}^{2}[t_{0},\,t_{f}]\Big)^{|\mathcal{P}|}
\end{equation*}%
where $\Big(\mathcal{L}_{+}^{2}[t_{0},\,t_{f}]\Big)^{|\mathcal{P}|}$ denotes
the non-negative cone of the $|\mathcal{P}|$-fold product of the Hilbert
space $\mathcal{L}^{2}[t_{0},\,t_{f}]$ of square-integrable functions on the
compact interval $[t_{0},\,t_{f}]$. The inner product of the Hilbert space $%
\Big(\mathcal{L}^{2}[t_{0},\,t_{f}]\Big)^{|\mathcal{P}|}$ is defined as 
\begin{equation}
\big<u,\,v\big>~\doteq ~\int_{t_{0}}^{t_{f}}\big(u(s)\big)^{T}\,v(s)\,ds
\label{nnorm}
\end{equation}%
where the superscript $T$ denotes transpose of vectors. Moreover, the norm 
\begin{equation}
\big\|u\big\|_{\mathcal{L}^{2}}~\doteq ~\big<u,\,u\big>^{1/2}  \label{l2norm}
\end{equation}%
is induced by the inner product (\ref{nnorm}).

Next, we need to consider a more general notion of travel cost that will
motivate on-time arrivals. To this end, for each $p\in \mathcal{P}$, we
introduce the effective unit path delay operator $\Psi
_{p}:[t_{0},\,t_{f}]\times \Big(\mathcal{L}_{+}^{2}[t_{0},\,t_{f}]\Big)^{|%
\mathcal{P}|}\rightarrow \mathcal{\Re }_{++}^{1}$ and define it as follows: 
\begin{equation}
\Psi _{p}(t,\,h)~ \doteq ~D_{p}(t,\,h)+\mathcal{F}\Big(t+D_{p}(t,\,h)-T_{A}%
\Big)  \label{phidef}
\end{equation}%
where $\mathcal{F}(\cdot )$ is the penalty for early or late arrival
relative to the target arrival time $T_{A}$. Note that, for
convenience, $T_{A}$ is assumed to be independent of destination. However,
that assumption is easy to relax, and the consequent generalization of our
model is a trivial extension. We interpret $\Psi _{p}(t,\,h)$ as the
perceived travel cost of driver starting at time $t$ on path $p$ under
travel conditions $h$. Presently, our only assumption on such costs is that
for each $h\in \Big(\mathcal{L}_{+}^{2}[t_{0},\,t_{f}]\Big)^{|\mathcal{P}|}
$, the vector function $\Psi (\cdot ,\,h):[t_{0},\,t_{f}]\rightarrow 
\mathcal{\Re }_{++}^{\left\vert \mathcal{P}\right\vert }$ is measurable and
strictly positive. The assumption of measurability was used for a measure
theory-based argument in \cite{Friesz1993}. Later in this paper, we shall
discuss other properties of this operator, such as continuity on a Hilbert
space. The continuity of effective delay is crucial for applying the general
theorems in \cite{Browder}, especially Theorem \ref{mainthm} stated above.

To support the development of a dynamic network user equilibrium model, we
introduce some additional constraints. Foremost among these are the flow
conservation constraints 
\begin{equation}
\sum_{p\in \mathcal{P}_{ij}}\int_{t_{0}}^{t_{f}}h_{p}(t)\,dt~=~Q_{ij}\qquad
\forall ~(i,\,j)\in \mathcal{W}  \label{flowcons}
\end{equation}%
where $\mathcal{P}_{ij}$ is the set of all paths that connect
origin-destination (O-D) pair $(i,\,j)\in \mathcal{W}$, while $\mathcal{W}$
is the set of all O-D pairs. In addition, $Q_{ij}$ is the fixed travel
demand for O-D pair $(i,\,j)$. Using the notation and concepts we have thus
far introduced, the set of feasible solutions for DUE when the effective
delay operator $\Psi (\cdot ,\,\cdot )$ is given is 
\begin{equation}
\Lambda ~=~\left\{ h\in \Big(\mathcal{L}_{+}^{2}[t_{0},\,t_{f}]\Big)^{|%
\mathcal{P}|}:\quad \sum_{p\in \mathcal{P}_{ij}}\int_{t_{0}}^{t_{f}}h_{p}(t)%
\,dt~=~Q_{ij}\quad \forall ~(i,\,j)\in \mathcal{W}\right\}   \label{feasible}
\end{equation}%
Using a presentation very similar to the above, the notion of a dynamic user
equilibrium in continuous time was first introduced by \cite{Friesz1993},
who employ a definition tantamount to the following:

\begin{definition}
\label{duedef}\textbf{(Dynamic user equilibrium)}. A vector of departure
rates (path flows) $h^{\ast }\in \Lambda $ is a dynamic user equilibrium if 
\begin{equation}
h_{p}^{\ast }\left( t\right) >0,p\in \mathcal{P}_{ij}\Longrightarrow \Psi
_{p}\left[ t,h^{\ast }\left( t\right) \right] =v_{ij}\in \Re _{++}^{1}\text{
\ \ \ }\forall (i,\,j)\in \mathcal{W}  \label{defdue}
\end{equation}%
We denote the dynamic user equilibrium defined this way by $DUE\big( \Psi
,\Lambda,\, [t_{0},\,t_{f}] \big) $.
\end{definition}

\noindent In the analysis to follow, we focus on the following
infinite-dimensional variational inequality formulation of the DUE problem
reported in Theorem 2 of \cite{Friesz1993}.

\begin{equation}
\left. 
\begin{array}{c}
\text{find }h^{\ast }\in \Lambda \text{ such that} \\ 
\sum\limits_{p\in \mathcal{P}}\displaystyle\int\nolimits_{t_{0}}^{t_{f}}\Psi
_{p}(t,h^{\ast })(h_{p}-h_{p}^{\ast })dt\geq 0 \\ 
\forall h\in \Lambda 
\end{array}%
\right\} VI\big(\Psi,\,\Lambda,\, [t_{0},\, t_{f}]\big)  \label{duevi}
\end{equation}%
The variational inequality formulation $VI\big(\Psi ,\,\Lambda
,[t_{0},\,t_{f}]\big)$ expressed above subsumes almost all DUE models
regardless of the arc dynamics or the network loading models employed.

\section{The dynamic network loading}

\label{DNL} A key ingredient of the variational inequality formulation of
the DUE (\ref{duevi}) is the effective delay operator $\Psi (t,\,\cdot )$,
which maps a vector of admissible departure rates to the vector of strictly positive
travel costs associated with each route-and-departure-time choice. The
problem of predicting time-varying network flows consistent with known
travel demands and departure rates (path flows) is usually referred to as
the \textit{dynamic network loading} (DNL) problem. Since effective path
delays are constructed from arc delays that depend on arc activity and
performance, DNL is intertwined with the determination of effective delay
operators.

In this section we present a continuous-time DNL model. This model is based
on a reformulation of the Vickrey model \citep{Vickrey}, which we call the 
\textit{generalized Vickrey model} (GVM); it was apparently first
proposed in \cite{HFY1, HFY2}. The generalized Vickrey model determines
arc exit flow and the arc traversal time from arc entry flow in an explicit way. This
formulation not only leads to a simple and explicit computational scheme,
but also makes it easier to conduct rigorous analyses of the arc delay
operator and, hence, of the effective path delay operator $\Psi (t,\,\cdot )$%
.

\subsection{The generalized Vickrey model}

First introduced in \cite{Vickrey}, the Vickrey model is based on two
key assumptions: (i) vehicles have negligible sizes, and, therefore, any non-empty 
queue is of negligible size; and (ii) link traversal time consists of a fixed
travel time plus a congestion-related arc-traversal delay. Let us introduce
the following notations: 
\begin{align*}
u(t):& \quad \hbox{link entering flow} \\
M~:& \quad \hbox{flow capacity of the bottleneck located at the exit of the link} \\
q(t):& \quad \hbox{queue size} \\
w(t):& \quad \hbox{link exit flow} \\
T~:& \quad \hbox{constant free flow travel time} \\
\lambda(t):& \quad \hbox{link traversal time when the time of entry is }~t
\end{align*}%
\noindent Then the model is described by the following set of equations. 
\begin{equation}
w(t)~=~%
\begin{cases}
\min \left\{ u(t-T),\,M\right\} \qquad  & q(t)~=~0 \\ 
M & q(t)~\neq ~0%
\end{cases}
\label{pqm1}
\end{equation}%
\begin{equation}
{\frac{dq(t)}{dt}}~=~u(t-T)-w(t)  \label{pqm2}
\end{equation}%
\begin{equation}
\lambda(t)~=~T+{\frac{q(t+T)}{M}}  \label{pqm3}
\end{equation}%
Notice that (\ref{pqm1}) and (\ref{pqm2}) amount to an ordinary differential
equation (ODE) with a right hand side that is discontinuous in the state
variable $q(\cdot )$. Such an ODE has been the main hurdle to further
analysis and computation of this model in continuous time. In \cite{HFY1, HFY2}, a reformulation of the Vickrey model as a Hamilton-Jacobi equation is
proposed and solved with a version of the Lax-Hopf formula (the reader is referred to \cite{Evans} for more details on  Hamilton-Jacobi equation and Lax-Hopf formula). As
a result, the solutions to (\ref{pqm1})-(\ref{pqm3}) are obtained in closed
form. Due to space limitation, we  omit further details of relevant analysis, and refer the reader to \cite{HFY1, HFY2}. 

Let us next introduce the cumulative entering vehicle count $U(\cdot )$ and
the exiting vehicle count $\mathcal{W}(\cdot )$ at the entrance and exit of
the link of interest, respectively. Furthermore, $U(\cdot )$ is assumed to
be non-decreasing and left-continuous. Notice that these latter assumption
imply that the link entry flows can be unbounded and possibly contain \ the
dirac-delta function. In contrast, Vickrey's original model requires
that the entry flow to be at least Lebesgue integrable. As such, the GVM is
more general than the Vickrey model.

Using the notation introduced previously, an equivalent statement of (\ref{pqm1}) through (\ref{pqm3}) is the following: 
\begin{align}
\mathcal{W}(t)& ~=~\min_{\tau \leq t-T}\left\{ U(\tau )-M\,\tau \right\}
+M(t-T)  \label{pqm4} \\
q(t)& ~=~U(t-T)-M(t-T)-\min_{\tau \leq t-T}\left\{ U(\tau )-M\,\tau \right\} \label{pqm5} \\
\lambda(t)& ~=~T+{1\over M} \left(U(t)-M(t)-\min_{\tau\leq t}\{U(\tau)-M\tau\}\right) \label{pqm6}
\end{align}%
Note that in the system (\ref{pqm4})-(\ref{pqm6}), all the variables of interest are
explicitly stated in terms of the cumulative entering vehicle count $U(\cdot )$. Identities (\ref{pqm4})-(\ref{pqm6}) will serve as the mathematical formulation
of link dynamics in the dynamic network loading sub-problem, as we shall explain shortly. The system (\ref%
{pqm4})-(\ref{pqm6}) may also be used for deriving mathematical properties
of the effective path delay operator, as is demonstrated in Section \ref%
{psicont}.

\subsection{The network model}

It is straightforward to extend
the generalized Vickrey model to a network, which is represented as a directed graph   $G(N,\,A)$, where
$N$ and $A$ are the set of nodes and arcs, respectively.  In order to proceed, we introduce some additional notations. In particular, for each node $v\in N$, let $
\mathcal{I}^{v}$ be the set of incoming links, $\mathcal{O}^{v}$ the set of
outgoing links. For each arc $a\in A$, let $u_a(t)$,  $w_a(t)$ be the entry flow and exit flow, respectively. The arc entry/exit flows are the sum of entry/exit flows associated with individual paths using this arc; that is, 
\begin{equation}\label{disaggregate1}
u_a(t)~=~\sum_{p\in\mathcal{P}} \delta_{ap}\,u_a^p(t), \qquad w_a(t)~=~\sum_{p\in\mathcal{P}} \delta_{ap}\,w_a^p(t)\qquad \forall~a\in A
\end{equation}
\noindent where 
$$
\delta_{ap}~=~\begin{cases} 1 \qquad &\hbox{if arc}~a~\hbox{belongs to path}~p\\
0\qquad & \hbox{otherwise}
\end{cases}
$$
In equation \eqref{disaggregate1}, we use $u_a^p(\cdot)$ to denote the link entering flow associated with path $p$, and $w_a^p(\cdot)$ to denote the link exiting flow associated with path $p$.   Let us also define the cumulative entering  vehicle count $U_a(t)$ and cumulative exiting vehicle count $W_a(t)$, for each arc $a$. Similarly, each one is disaggregated into quantities associated with each path that uses this arc:
\begin{equation}\label{disaggregate2}
U_a(t)~=~\sum_{p\in\mathcal{P}} \delta_{ap}\,U_a^p(t), \qquad W_a(t)~=~\sum_{p\in\mathcal{P}} \delta_{ap}\,W_a^p(t)\qquad \forall~a\in A
\end{equation}

The arc traversal time function, denoted $\lambda_a(\cdot)$, is the time taken to traverse arc $a$ when the time of entry is $t$. The arc exit time function $\tau_a(t)$ is defined as $\tau_a(t) \doteq t+\lambda_a(t)$; that is, $\tau_a(t)$ represents the time a car leaves arc $a$ when the time of its entry is $t$.

 For each group of drivers using the same arc $a$, the ratio of their arrival and departure rates must be the same under {\it first-in-first-out} (FIFO). This is expressed as 
\begin{equation}
w^{p}_{a}\big(\tau_{a}(t)\big)~=~%
\begin{cases}
\displaystyle w_{a}\big(\tau_{a}(t)\big)\cdot {\frac{u^{p}_{a}(t)}{u_{a}(t)}%
}\qquad  & \hbox{if}~~u_{a}(t)~\neq ~0 \\ 
0\qquad  & \hbox{if}~~u_{a}(t)~=~0%
\end{cases}
\label{junction3}
\end{equation}%
(\ref{junction3}) uniquely determines the turning percentages at junctions
with more than one outgoing links, and is consistent with the FIFO discipline  and
established route choices. It remains to express the path delay as the sum
of finitely many link delays. If we describe path $p\in\mathcal{P}$ as the following sequence of conveniently labeled arcs:
$$
p~=~\left\{a_1,\,a_2,\,\ldots,\,a_{i-1},\,a_i,\,a_{i+1}\,\ldots,\,a_{m(p)}\right\}
$$
where $m(p)$ is number of arcs in path $p$.

It then follows immediately that the arrival time along path $p$, when the departure time at the origin is $t$, can be expressed as a composition of arc exit time functions: 
\begin{equation}\label{pexit}
\tau_{p}(t)~=~\tau_{a_{m(p)}}\,\circ \,\ldots\, \circ \tau_{a_2} \,\circ\,\tau_{a_1}(t)
\qquad p~=~\left\{ a_{1},\,a_{2},\,\ldots
,\,a_{m(p)}\right\} \in \mathcal{P}  
\end{equation}%
where the operator $\circ$ means composition, that is, $f\circ g(x)\doteq f\big(g(x)\big)$.

Now the complete network loading procedure is given by (\ref{pqm4})-(\ref%
{pexit}), which is interpreted as a well defined \textit{differential
algebraic equation} (DAE) system. Moreover, as well shall see in the next
section, the (effective) path delay operator defined in this way is strongly
continuous from the subset $\Lambda\subset \Big(L^2[t_0,\,t_f]\Big)^{|\mathcal{P}|}$  into $\Big(\mathcal{%
L}^{2}[t_{0},\,t_{f}]\Big)^{|\mathcal{P}|}$.

%
%

\section{Existence of the DUE}

\label{dueexistence} Existence results for DUE are most general if based on
formulation (\ref{duevi}). Theorem \ref{mainthm} for the existence of
solutions of variational inequalities in topological spaces can be applied
if the operator $\Psi(t,\,\cdot )$ can be shown to be continuous and
the feasible set $\Lambda $ can be shown to be compact. After Section \ref%
{psicont} addresses the continuity of the effective delay operator, based on
the DNL model introduced previously, the last obstacle to proving
existence is the compactness of $\Lambda $, which unfortunately does not
generally occur in SRDC-DUE. To overcome the
aforementioned difficulty, we will consider instead successive
finite-dimensional approximations of $\Lambda $, and rely on a topological
argument. Such an approach is mathematically rigorous but much more
challenging than would be the case if $\Lambda $ were compact in the
appropriate Hilbert space. The topological argument and supporting
infrastructure for a proof of existence are presented in Section \ref{bddsec}
and Section \ref{visec}.

\subsection{\label{psicont}Continuity of the effective path delay operator}

In this section, we will establish continuity of the map $h\mapsto \Psi
(\cdot ,\,h)$. These results will be essential for the proof of existence
theorem for DUE in Section \ref{visec}. Notice that unlike the argument in 
\cite{ZM} which requires {\it a priori} bound for the path flows, the
proof provided here works for unbounded path flows and even distributions,
thanks to the generalized Vickrey model. 

The next lemma provides a sufficient condition for the continuity of the delay function $\lambda_a(\cdot),\, a\in A$. 

\begin{lemma}\label{continuitylemma}
Consider an arc $a\in A$, with inflow $u_a(\cdot)$. Under the generalized Vickrey model expressed in (\ref{pqm4})-(\ref{pqm6}), the arc delay function $\lambda_a(\cdot)$ is continuous if  $u_a(\cdot)\in \mathcal{L}^2[t_0,\,t_f]$.
\end{lemma}
\begin{proof}
Assume that $u_a(\cdot)\in\mathcal{L}^2[t_0,\,t_f]$, then $u_a(\cdot)\in\mathcal{L}^1[t_0,\,t_f]$. Therefore the cumulative entering vehicle count 
$$
U_a(t)~\doteq~\int_{t_0}^tu_a(s)\,ds
$$
is absolutely continuous. It is straightforward to verify that the following quantity is continuous. 
$$
q_a(t)~\doteq~U_a(t)-M_a t-\min_{\tau\leq t-T_a}\left\{U_a(\tau)-M_a\tau \right\}
$$
where $q_a(t)$ denotes the queue length, $M_a$ denotes the bottleneck capacity and $T_a$ denotes the constant free flow time. By (\ref{pqm6}), the function $\lambda_a(\cdot)$ is continuous. 
\end{proof}

The next lemma is a technical result that will facilitate the proof of Theorem \ref{contthm}.

\begin{lemma}\label{techlemma}
Let $g_n(\cdot): [a_1,\,b_1]\rightarrow [a_2,\,b_2],\,n\geq 1$ be a sequence of functions such that $g_n$ converges to $g(\cdot): [a_1,\,b_1]\rightarrow [a_2,\,b_2]$ uniformly. In addition, assume $f(\cdot): [a_2,\,b_2]\rightarrow \mathcal{\Re}^1$ is continuous.  Then the following uniform convergence holds.
$$
f\big(g_n(\cdot)\big)~\longrightarrow ~ f\big(g(\cdot)\big)\qquad n~\longrightarrow~\infty
$$
\end{lemma}
\begin{proof}
According to the Heine-Cantor theorem \citep{Royden}, $f(\cdot)$ is uniformly continuous on $[a_2,\,b_2]$. It follows that, for every $\varepsilon>0$, there exists $\delta>0$ such that for any  $y_1,\,y_2\in[a_2,\,b_2]$, whenever $|y_1-y_2|<\delta$, the inequality 
$$
|f(y_1)-f(y_2)|~\leq~\varepsilon
$$
holds. Moreover, by the uniform convergence of $g_n$, there exists some $N>0$ such that, for all $n>N$, we have
$$
|g_n(x)-g(x)|~<~\delta\qquad\forall~x\in[a_1,\,b_1]
$$
Thus, for every $n> N$, 
$$
\big|f\big(g_n(x)\big)-f\big(g(x)\big)\big|~\leq~\varepsilon\qquad \forall~x\in[a_1,\,b_1]
$$
\end{proof}

\begin{theorem}
\label{contthm} Under the network loading model described in Section \ref%
{DNL}, the effective path delay operator $\Psi (t,\,\cdot ):\Lambda
\rightarrow \big(\mathcal{L}^{2}[t_{0},\,t_{f}]\big)^{|\mathcal{P}|}$, $%
h\mapsto \Psi (\cdot ,\,h)$ is well-defined and continuous.
\end{theorem}

\begin{proof}
For each $h\in \Lambda $, the functions $\Psi _{p}(\cdot ,\,h),\,p\in \mathcal{P}$ are uniquely determined by the network loading procedure.
To show that the effective path delay operator is well-defined, it remains
to show that $\Psi (\cdot ,\,h)\in \Big(\mathcal{L}^{2}[t_{0},\,t_{f}]\Big)%
^{|\mathcal{P}|}$ for each $h\in \Lambda $. Notice that there exists an
upper bound for the path delays regardless of the network flow profile: 
\begin{equation}\label{Dub}
D_{p}(t,\,h)~\leq ~\sum_{a\in p}\left\{ {\frac{1}{M_{a}}}\sum_{(i,j)\in 
\mathcal{W}}Q_{ij}+T_{a}\right\} \qquad \forall h\in \Lambda ,\, \forall p\in 
\mathcal{P},\,\forall t\in [t_{0},\,t_{f}]
\end{equation}
where $M_a$, $T_a$ are the bottleneck capacity and the free flow time respectively, that are associated with arc $a$. Recall the definition of the effective path delay (\ref{phidef}):
$$
\Psi _{p}(t,\,h)~= ~D_{p}(t,\,h)+\mathcal{F}\Big(t+D_{p}(t,\,h)-T_{A}%
\Big) 
$$
Since $\mathcal{F}(\cdot)$ is continuous, the uniform boundedness of $D_p(t,\,h)$, as shown in \eqref{Dub}, thus implies the uniform boundedness of  $\Psi _{p}(t,\,h)$ for all $h\in\Lambda, \, p\in \mathcal{P}$ and $t\in[t_0,\,t_f]$. This leads to the conclusion that $\Psi (\cdot ,\,h)\in \Big(\mathcal{L}^{2}[t_{0},\,t_{f}]\Big)^{|\mathcal{P}|}$ for all $h\in \Lambda $. 

With the preceding as background, the proof of continuity of the effective delay
operator may be given in five parts.

\noindent \textbf{Part 1.} We first focus on a single link $a$. For the convenience of notations, the subscript $a$ will be dropped for now. Consider a sequence
of entering flows $u_{\nu },\,\nu \geq 1$ that converge to $u$ in the $%
\mathcal{L}^{2}$-norm, that is,
\begin{equation*}
\Vert u_{\nu }-u\Vert _{2}~\doteq ~\left( \int_{t_{0}}^{t_{f}}\left( u_{\nu
}(t)-u(t)\right) ^{2}\,dt\right) ^{1/2}~\longrightarrow ~0\text{ \ \ as \ \ }%
\nu ~\longrightarrow ~\infty 
\end{equation*}%
Consider the cumulative entering vehicle counts 
\begin{equation*}
\begin{cases}
\displaystyle U_{\nu }(t)~\doteq ~\displaystyle \int_{t_{0}}^{t}u_{\nu }(s)\,ds\qquad
\nu~\geq ~1 \\
\displaystyle U(t) ~\doteq ~\displaystyle \int_{t_{0}}^{t}u(s)\,ds
\end{cases}
\qquad t\in  [t_{0},\,t_{f}]
\end{equation*}%
 Then $U_{\nu }$ converge to $U$ uniformly on $%
[t_{0},\,t_{f}]$: this is due to the following simple observation 
\begin{equation}\label{L2L1}
\left\vert U_{\nu }(t)-U(t)\right\vert ~\leq ~\int_{t_{0}}^{t}\left\vert
u_{\nu }(s)-u(s)\right\vert \,ds~\leq ~\Vert u_{\nu }-u\Vert _{1}~\leq
~(t_{0}-t_{f})^{1/2}\Vert u_{\nu }-u\Vert _{2}~\longrightarrow ~0
\end{equation}
where $\Vert \cdot \Vert _{1}$ is the norm in $\mathcal{L}^{1}[t_{0},\,t_{f}]
$. The last inequality of \eqref{L2L1} is a version of Jenssen's inequality.

\noindent \textbf{Part 2.} Define $R(\tau )\doteq U(\tau )-M \tau $, $%
R_{\nu }(\tau )\doteq U_{\nu }(\tau )-M\tau,\,\tau\in[t_0,\,t_f]$, where $M$ is the bottleneck capacity. We claim the following uniform convergence:
\begin{equation}
\min_{\tau \leq t}\left\{ R_{\nu }(\tau )\right\} ~\longrightarrow
~\min_{\tau \leq t}\left\{ R(\tau )\right\} \qquad \forall ~t\in \lbrack
t_{0},\,t_{f}]  \label{contproof1}
\end{equation}%
Indeed, for any $\varepsilon >0$, by the uniform convergence of $U_{\nu},\,\nu\geq 1$, we can choose $N$ such that for all $\nu
\geq N$, the following inequality holds 
\begin{equation*}
\left\vert U_{\nu }(t)-U(t)\right\vert ~\leq ~\varepsilon \qquad \forall
~t\in \lbrack t_{0},\,t_{f}]
\end{equation*}%
Fix any $t$, if $\nu \geq N$, then
\begin{equation}
\left\vert R_{\nu }(\tau )-R(\tau )\right\vert ~=~\left\vert U_{\nu }(\tau
)-U(\tau )\right\vert ~\leq ~\varepsilon \qquad\forall \tau \in[t_0,\,t_f]   \label{contproof2}
\end{equation}
Define $\hat{\tau}=\hbox{argmin}_{\tau \leq t}\left\{ R(\tau )\right\}$. By (\ref{contproof2}) we have 
\begin{equation}
\min_{\tau \leq t}\left\{ R_{\nu }(\tau )\right\} ~\leq ~R_{\nu }(\hat{\tau}%
)~\leq ~R(\hat{\tau})+\varepsilon ~=~\min_{\tau \leq t}\left\{ R(\tau
)\right\} +\varepsilon   \label{contproof3}
\end{equation}%
On the other hand, define $\hat{\tau}_{\nu }=\hbox{argmin}_{\tau \leq t}\left\{
R_{\nu }(\tau )\right\}$ for each $\nu\geq 1$. Then given $\nu \geq N$, it must hold that 
\begin{equation}
\min_{\tau \leq t}\left\{ R(\tau )\right\} ~\leq ~R(\hat{\tau}_{\nu })~\leq
~R^{(\nu )}(\hat{\tau}_{\nu })+\varepsilon ~=~\min_{\tau \leq t}\left\{
R_{\nu }(\tau )\right\} +\varepsilon   \label{contproof4}
\end{equation}%
Taken together, (\ref{contproof3}) and (\ref{contproof4}) imply 
\begin{equation*}
\left\vert \min_{\tau \leq t}\left\{ R_{\nu }\right\} -\min_{\tau \leq
t}\left\{ R(\tau )\right\} \right\vert ~\leq ~\varepsilon \qquad \forall
~\nu \geq N
\end{equation*}%
Since $t$ is arbitrary, the claim is demonstrated.

\noindent \textbf{Part 3.} An immediate consequence of {\bf Part 2} and (\ref%
{pqm4})-(\ref{pqm6}) is the following uniform convergence 
\begin{equation}
W_{\nu }(t)~\longrightarrow ~W(t),\quad q_{\nu }(t)~\longrightarrow ~q(t),\quad
\lambda_{\nu }(t)~\longrightarrow~\lambda (t),\quad \tau _{\nu
}(t)~\longrightarrow ~\tau (t)\qquad \nu ~\longrightarrow ~\infty 
\label{contproof5}
\end{equation}%
for which we employ notation whose meaning is transparent. The next step is
to extend such convergence to the whole network. Consider the sequence of
departure rates $h_{\nu }$ converging to $h$ in the $\Vert \cdot \Vert _{%
\mathcal{L}^{2}}$ norm. By the definition (\ref{l2norm}), this implies each
path flow $h_{p, \nu}(\cdot )\rightarrow h_{p}$ in the $\Vert \cdot \Vert
_{2}$ norm, for all $p\in \mathcal{P}$. A simple induction based on results
established in {\bf Part 2} yields, as $\nu\longrightarrow \infty$, 
\begin{equation}\label{aconv}
U_{a, \nu }(t)~\longrightarrow ~U_{a}(t),\quad W_{a, \nu
}(t)~\longrightarrow ~W_{a}(t),\quad D _{a, \nu
}(t)~\longrightarrow ~D_{a}(t),\quad \tau _{a, \nu
}(t)~\longrightarrow ~\tau_{a}(t)
\end{equation}%
uniformly for all $a\in A$.

\noindent{\bf Part 4.}
We will show next the uniform convergence of the path delay function $D_p(\cdot,\,h_{\nu})\rightarrow D_p(\cdot,\,h)$, based on (\ref{aconv}). Recall the path exit time function (\ref{pexit})
\begin{equation}\label{induction}
\tau_{p}(t)~=~\tau_{a_{m(p)}}\,\circ \,\ldots\, \circ \tau_{a_2} \,\circ\,\tau_{a_1}(t)
\qquad p~=~\left\{ a_{1},\,a_{2},\,\ldots,\, a_{m(p)}\right\}\in\mathcal{P}
\end{equation}
We start by showing that $\tau_{a_2,\nu}\circ\tau_{a_1, \nu}(t)\rightarrow \tau_{a_2}\circ\tau_{a_1}(t)$ uniformly.

For every $\nu\geq 1$, since the inflow of arc $a_2$ is square-integrable, $\tau_{a_2,\nu}(\cdot)$ is continuous by Lemma \ref{continuitylemma}. This means that $\tau_{a_2}(\cdot)$ is also continuous since it is the uniform limit of $\tau_{a_2,\nu}(\cdot)$. Lemma \ref{techlemma} then implies that $\tau_{a_2}\big(\tau_{a_1,\nu}(\cdot)\big)$ converges uniformly to $\tau_{a_2}\big(\tau_{a_1}(\cdot)\big)$, that is, for any $\varepsilon>0$, there exists an $N_1>0$ such that for all $\nu> N_1$, 
$$
\big|\tau_{a_2}\big(\tau_{a_1,\nu}(t)\big)-\tau_{a_2}\big(\tau_{a_1}(t)\big)\big|~<~\varepsilon/2\qquad\forall~t\in[t_0,\,t_f]
$$
Moreover, there exists some $N_2>0$ such that for all $\nu> N_2$, 
$$
\big|\tau_{a_2,\nu}(t)-\tau_{a_2}(t)\big|~<~\varepsilon/2\qquad\forall~t\in[t_0,\,t_f]
$$
Now let $N_0=\max\{N_1,\,N_2\}$. Then for any $\nu>N_0$ and any $t\in[t_0,\,t_f]$,
\begin{align*}
&\big|\tau_{a_2,\nu}\big(\tau_{a_1,\nu}(t)\big)-\tau_{a_2}\big(\tau_{a_1}(t)\big)\big|\\
~\leq~& \big|\tau_{a_2,\nu}\big(\tau_{a_1,\nu}(t)\big)-\tau_{a_2}\big(\tau_{a_1,\nu}(t)\big)\big|+\big|\tau_{a_2}\big(\tau_{a_1,\nu}(t)\big)-\tau_{a_2}\big(\tau_{a_1}(t)\big)\big|\\
~<~& \varepsilon/2+\varepsilon/2~=~\varepsilon
\end{align*}
This shows the desired  uniform convergence $\tau_{a_2,\nu}\circ\tau_{a_1, \nu}(t)\rightarrow \tau_{a_2}\circ\tau_{a_1}(t)$.  

The uniform convergence $\tau_{p, \nu}(\cdot)\rightarrow \tau_p(\cdot)$ follows immediately by (\ref{induction}) and mathematical induction with Lemma \ref{techlemma} . As a result, we obtain the uniform convergence of path delay
$$
D_p(\cdot,\,h_{\nu})~\longrightarrow~D_p(\cdot,\,h)  \qquad \nu~\longrightarrow~\infty
$$

\noindent\textbf{Part 5}. Finally, recall the definition of the effective delay
\begin{equation*}
\Psi (t,\,h)~=~D_{p}(t,\,h)+\mathcal{F}\big(t+D_{p}(t,\,h)-T_{A}\big)
\end{equation*}%
Note that $\mathcal{F}(\cdot)$ is continuous, the following uniform convergence follows by Lemma \ref{techlemma}
$$
\mathcal{F}\big(t+D_p(t,\,h_{\nu})-T_A\big)~\longrightarrow~\mathcal{F}\big(t+D_p(t,\,h)-T_A\big)  \qquad \nu~\longrightarrow~\infty
$$
We conclude that the effective delay $\Psi_p(\cdot,\,h_{\nu})$ converges uniformly  to $\Psi_p(\cdot,\,h)$. The
desired convergence in the $\Vert \cdot \Vert _{\mathcal{L}^{2}}$ norm now
follows since the interval $[t_{0},\,t_{f}]$ is compact.
\end{proof}

\subsection{Alternative definition of effective path delay}\label{bddsec} 

The integrals employed in defining the feasible domain (\ref%
{feasible}) are not enough to assure bounded path flows $h_{p},\,p\in 
\mathcal{P}$. This observation is the fundamental hurdle to providing
existence of the DUE solution. One of the main accomplishments of this paper
is to address the boundedness of path flows not only for the proof of
existence result but also for future analysis and estimation of network
flows. In this section, we will present an alternative formulation of the
effective path delay $\Psi _{p}(t,\,h)$, where that alternative formulation
will facilitate our analysis leading to the proof of our main result,
Theorem \ref{VIthm}.

As a motivation, let us recall the effective delay operator 
\begin{equation}
\Psi _{p}(t,\,h)~\doteq ~D_{p}(t,\,h)+\mathcal{F}\big(t+D_{p}(t,\,h)-T_{A}%
\big)  \label{ed}
\end{equation}%
 In order to simplify our analysis, it is convenient to rewrite (\ref{ed}) in a slightly different
form. In particular, for each O-D pair $(i,\,j)\in \mathcal{W}$, let us
introduce the cost function $\phi _{ij}(\cdot ):[t_{0},\,t_{f}]\rightarrow 
\mathcal{\Re }_{+}^{1}$, which is a function of departure time, and $\,\psi _{ij}(\cdot ):[t_{0},\,t_{f}]\rightarrow \mathcal{\Re }_{+}^{1}$, which is a function of arrival time. As we shall explain below, the users' travel costs can be alternatively expressed using functions $\phi_{ij}(\cdot)$ and $\psi_{ij}(\cdot)$.  Given any origin-destination pair $(i,\,j)\in\mathcal{W}$, and any driver who departs from the origin at time $t_d$, and arrives at destination at $t_a$, his/her travel cost is expressed as $\phi_{ij}(t_d)+\psi_{ij}(t_a)$. 

Fix any vector of path flows $%
h\in \Lambda $, recall the path exit time function $\tau_{p}(t)=t+D_p(t,\,h)$ where $t$ denotes departure time. Then (\ref{ed}) can be equivalently written as 
\begin{equation}\label{reveqn1}
\Psi _{p}(t,\,h)~=~ -t+\tau_p(t)+\mathcal{F}\big(\tau_p(t)-T_{A}%
\big)~=~ \phi _{ij}(t)+\psi _{ij}\big(\tau _{p}(t,\,h)\big)
\end{equation}
where 
\begin{equation}\label{vppsidef}
\phi _{ij}(t)~\doteq ~-t,\qquad \psi _{ij}(\tau_p(t))~\doteq ~\tau_p(t)+\mathcal{F}\big(%
\tau_p(t)-T_{A}\big)  
\end{equation}

\begin{remark}
In \cite{BH, BH1, BH2}, the unit travel cost is measured in terms of $\phi_{ij}(\cdot)$ and $\psi_{ij}(\cdot)$. In other words, the general effective
delay (\ref{ed}) can be alternatively evaluated as a sum of costs at the
beginning and at the end of each driver's trip.
\end{remark}

\noindent In Section \ref{visec} that follows, we will exploit this alternative
representative of effective path delay (cost) to establish existence. To prepare for the existence proof, we consider a general network $G(N, \, A)$, and associate to each O-D pair $(i,\,j)$ the pair of cost functions $\phi _{ij}(\cdot )$ and $\psi _{ij}(\cdot )$. We make the following two assumptions regarding $\phi_{ij}(\cdot )$ and $\psi _{ij}(\cdot )$ and the underlying link performance model.\\

\noindent {\bf A1.}  For each $(i,\,j)\in\mathcal{W}$, $\phi _{ij}(\cdot )$ and $\psi
_{ij}(\cdot )$ are continuous on $[t_{0},\,t_{f}]$. Moreover, $\phi_{ij}(\cdot)$ is monotonically decreasing while $\psi_{ij}(\cdot)$ is monotonically increasing. In
addition, we assume that $\phi_{ij}(\cdot)$ is Lipschitz continuous with constant $L_{ij}$; and there exists $\Delta_{ij}>0$ such that
\begin{equation}
\psi_{ij}(t_2)-\psi_{ij}(t_1)~\geq~\Delta_{ij}(t_2-t_1)  \qquad \forall~t_0\leq t_1<t_2\leq t_f
\label{a1eqn}
\end{equation}%

\noindent {\bf A2.}  Each link $a\in A$ of the network has
a finite exit flow capacity $M_{a}~<~\infty $. \\

\noindent Inequality (\ref{a1eqn}) requires that the arrival cost function $\psi _{ij}(\cdot )$
is strictly increasing, and the rate of increase is bounded below by $\Delta_{ij}$. In the case where $\psi_{ij}(\cdot)$ is continuously differentiable, this assumption is equivalent to requiring that ${d\over dt}\psi_{ij}(t)\geq \Delta_{ij}>0~~ t\in[t_0,\,t_f]$, which is further equivalent to ${d\over dt}\psi_{ij}(t)>0~~t\in[t_0,\,t_f]$, due to compactness.  As a special case, given the effective delay of the form (\ref{reveqn1}) and (\ref{vppsidef}),  {\bf A1} reduces to the following assumptions.\\

\noindent {\bf A1'.} $\mathcal{F}(\cdot )$ is
continuous on $[t_0,\,t_f]$ and satisfies 
$$
\mathcal{F}(t_2)-\mathcal{F}(t_1)~\geq~\Delta (t_2-t_1) \qquad \forall t_0~\leq~t_1~<~t_2~\leq~t_f
$$
for some $\Delta >-1$.\\

\noindent Assumption {\bf A2} applies to all link dynamics that impose an exit flow capacity; examples include the Vickrey model \citep{Vickrey} and the Lighthill-Whitham-Richards model  \citep{LW, Richards, LKWM}.

\begin{remark}
Our proposed formulation of travel cost subsumes another well-known class of travel cost functions that are employed in, e.g. \cite{Pang} and \cite{Yao}. Namely, given positive constants $\alpha,\,\beta$ and $\gamma$, the travel cost of road user is expressed as
\begin{equation}\label{abgmodel}
\alpha(t_a-t_d)+\begin{cases} \beta (T-t_a) \qquad t_a\leq T_A\\ 
\gamma (t_a-T) \qquad t_a> T_A\end{cases}
\end{equation}
where $t_d,\,t_a$ denote departure and arrival times, respectively. $T_A$ is the desired arrival time. It is assumed that $\gamma>\alpha>\beta$.  The expression in \eqref{abgmodel} can be  rewritten as
\begin{align*}
\alpha(t_a-t_d)+\begin{cases} \beta (T_A-t_a) \qquad t_a\leq T_A\\ 
\gamma (t_a-T_A) \qquad t_a> T_A\end{cases}&~=~-\alpha\, t_d+\begin{cases}(\alpha-\beta)t_a+\beta T_A\qquad t_a\leq T_A\\
(\alpha+\gamma)t_a-\gamma T_A\qquad t_a>T_A\end{cases} \\&~=~\phi(t_d)+\psi(t_a)
\end{align*}
where 
$$
\phi(t_d)~\doteq~-\alpha\,t_d,\qquad \psi(t_a)~\doteq~\begin{cases}(\alpha-\beta)t_a+\beta T_A\qquad t_a\leq T_A\\
(\alpha+\gamma)t_a-\gamma T_A\qquad t_a>T_A\end{cases}
$$
One can easily check that such $\phi(\cdot)$ and $\psi(\cdot)$ satisfy assumption {\bf A1}. 
\end{remark}

In view of the preceding assumptions {\bf A1} and {\bf A2}, we are prompted to define the following: 
\begin{equation}\label{phimaxdef}
\phi_{max}'~\doteq~\max_{(i,\,j)\in\mathcal{W}}L_{ij}~>~0
\end{equation}
\begin{equation}
\psi _{min}^{\prime }~\doteq ~\min_{(i,\,j)\in \mathcal{W}}\Delta_{ij}~>~0  \label{psimindef}
\end{equation}
\begin{equation}
M^{max}~\doteq ~\max_{a\in A}M_{a}~<~+\infty 
\label{fmaxdef}
\end{equation}%

\subsection{\label{visec}Existence of solution to the variational inequality}

The classical result explained by Theorem \ref{mainthm} will
be the key ingredient for the proof of existence of the DUE solution. Using
the same notation as in Theorem \ref{mainthm}, the underlying topological
vector space $E$ will be instantiated by $\big(\mathcal{L}^{2}[t_{0},\,t_{f}]\big)^{|\mathcal{P}|}$, which is a locally convex
topological vector space. The dual space $E^{\ast }$ will be again $\big(%
\mathcal{L}^{2}[t_{0},\,t_{f}]\big)^{|\mathcal{P}|}$.

\begin{theorem}
\label{VIthm}\textbf{(Existence of DUE)} Let assumptions {\bf A1} and {\bf A2} hold.
In addition, assume that the effective delay operator $\Psi: \Lambda\rightarrow \big(\mathcal{L}^2[t_0,\,t_f]\big)^{|\mathcal{P}|}$ is continuous. Then the dynamic user equilibrium problem as in Definition \ref{duedef} has a solution.
\end{theorem}
\begin{proof}
The proof is divided into four parts.

\noindent \textbf{Part 1.} Our strategy for demonstrating existence is to adapt
Theorem \ref{mainthm} to the locally convex topological vector space $\big(\mathcal{L}^{2}[t_{0},\,t_{f}]\big)^{|\mathcal{P}|}$, and its subset $\Lambda $. By assumption, the map $h\mapsto \Psi (\cdot ,\,h)$ is continuous from $\Lambda $ to the space of $\big(\mathcal{L}^{2}[t_{0},\,t_{f}]\big)^{|\mathcal{P}|}$. If $\Lambda $ were compact and
convex, we would immediately demonstrate the desired result. However, $%
\Lambda $ is bounded, closed and convex, but not compact in $\big(\mathcal{L}^{2}[t_{0},\,t_{f}]\big)^{|\mathcal{P}|}$.

\noindent \textbf{Part 2.} We will instead employ finite-dimensional approximations of 
$\Lambda $. In order to proceed, consider for each $n\geq 1$ the uniform
partition of interval $[t_{0},\,t_{f}]$ with $2^{n}$ sub-intervals \footnote{The choice of the number of sub-intervals is quite flexible: as long as the size of the sub-intervals approaches zero as $n\rightarrow+\infty$, our framework for showing existence will work.}
\begin{align*}
t_{0}~& =~t^{0}~<~t^{1}~<~t^{2}~\ldots ~<~t^{2^{n}}~=~t_{f} \\
t^{i}-t^{i-1}~& =~{\frac{t_{f}-t_{0}}{2^{n}}}\qquad i=1,\ldots ,2^{n}
\end{align*}%
Then consider the following sequence of finite-dimensional subsets 
\begin{equation}
\Lambda _{n}~\doteq ~\left\{ h\in \Lambda :\quad h_{p}(\cdot )\hbox{ is
constant on }\lbrack t^{i-1},\,t^{i}) \qquad p\in \mathcal{P}\right\}
~\subset~\Lambda  \qquad n\geq 1  \label{subspace}
\end{equation}%
We claim that for each $n\geq 1$, $\Lambda _{n}$ is compact and convex in $%
\big(\mathcal{L}^{2}[t_{0},\,t_{f}]\big)^{|\mathcal{P}|}$. Indeed, given
any $h^{n,1},\,h^{n,2}\in \Lambda _{n}$, and $\alpha \in \lbrack 0,\,1]$, $%
\alpha \,h^{n,1}+(1-\alpha )\,h^{n,2}$ is clearly nonnegative and constant
on each $[t^{i-1},\,t^{i}]$ for$\,i=1,\ldots ,2^{n}$. In addition, for any
origin-destination pair $(i,\,j)\in \mathcal{W}$, by definition (\ref%
{feasible}), 
\begin{equation*}
\sum_{p\in \mathcal{P}_{ij}}\int_{t_{0}}^{t_{f}}\alpha
\,h^{n,1}(t)+(1-\alpha )\,h^{n,2}(t)\,dt~=~\alpha \,Q_{ij}+(1-\alpha
)\,Q_{ij}~=~Q_{ij}
\end{equation*}%
This verifies that $\Lambda _{n}$ is convex. To see the compactness of $\Lambda_n$, we define
the map $\mu :\Lambda _{n}\rightarrow \mathcal{\Re }_{+}^{2^{n}\times |%
\mathcal{P}|},\,h\mapsto (a_{i,\,p}:\,1\leq i\leq 2^{n},\,\,p\in \mathcal{P})
$ where each vector $(a_{1,p},\ldots ,a_{2^{n},p})$ is the coordinate of $%
h_{p}$ under the natural basis $\{e_{n}^{i}\}_{i=1}^{2^{n}}$, where 
\begin{equation*}
e_{n}^{i}(t)~=~%
\begin{cases}
1\qquad  & t\in \lbrack t^{i-1},\,t^{i}) \\ 
0\qquad  & \hbox{else}%
\end{cases}%
\end{equation*}%
Clear, the map $\mu $ is one-to-one. Now fix $n$ and consider an arbitrary sequence $\{h^{n,\nu}\}_{\nu \geq 1}\subset \Lambda _{n}$, as well as the sequence of their
images $\big\{\mu \big(h^{n, \nu}\big)\big\}_{\nu \geq 1}\subset 
\mathcal{\Re }_{+}^{2^{n}\times |\mathcal{P}|}$. By (\ref{feasible}), the latter sequence is uniformly bounded by the following
quantity 
\begin{equation*}
\max_{(i,j)\in\mathcal{W}}{\frac{2^{n}\,Q_{ij}}{t_{f}-t_{0}}}
\end{equation*}%
According to the Bolzano-Weierstrass theorem, there exists a convergent subsequence $\big\{\mu (h^{n, \nu\,'})\big\}_{\nu\,' \geq 1}$ with limit $y^*\in\Re_+^{2^n\times|\mathcal{P}|}$. By
construction, the corresponding subsequence $\big\{h^{n, \nu\,')}\big\}_{\nu
\,'\geq 1}$ must converge uniformly to $\hat h^n\doteq \mu^{-1}(y^*)$. In view
of the compact interval $[t_{0},\,t_{f}]$, we conclude that this convergence
holds also in the $\Vert \cdot \Vert _{\mathcal{L}^{2}}$ norm. It now remains
to show $\hat{h}^{n}\in \Lambda _{n}$. Clearly $\hat{h}^{n}\geq 0$ and is
constant on the sub-intervals $[t^{i-1},\,t^{i}),\,i=1,\ldots ,2^{n}$.
Moreover, since $h^{n, \nu\,'}$ are uniformly bounded, by the
Dominated Convergence Theorem, 
\begin{equation*}
Q_{ij}~=~\lim_{\nu \,^{\prime }\rightarrow \infty }\sum_{p\in \mathcal{P}_{ij}}\int_{t_{0}}^{t_{f}}h_{p}^{n, \nu\,'}(t)\,dt~=~\sum_{p\in 
\mathcal{P}_{ij}}\int_{t_{0}}^{t_{f}}\hat{h}_{p}^{n}(t)\,dt  \qquad \forall (i,\,j)\in\mathcal{W}
\end{equation*}
This implies $\hat{h}^{n}\in \Lambda _{n}$. Thereby the sequential compactness, and hence the compactness, of $\Lambda_n$ is substantiated.

\noindent \textbf{Part 3.} For each $n\geq 1$, apply Theorem \ref{mainthm} to $\Lambda
_{n}$ and obtain $h^{n,\ast }\in \Lambda _{n}$ such that 
\begin{equation}
\Big<\Psi \big(\cdot ,\,h^{n,\ast }\big),\,h^{n}(\cdot )-h^{n,*}(\cdot )%
\Big>~\geq ~0  \qquad \forall ~h^{n}\in \Lambda _{n}  \label{finite}
\end{equation}%
Where the $\big<\,,\,\big>$ is the inner product (duality) defined in (\ref%
{nnorm}). Inequality (\ref{finite}) implies that given any $p\in\mathcal{P}$, if $%
h_{p}^{n,\ast }(t)>0,~t\in \lbrack t^{j},\,t^{j+1})$, then the following is true.
\begin{equation}
\int_{t^{j}}^{t^{j+1}}\Psi_p (t,\,h^{n,\ast })\,dt~=~\min_{0\leq k\leq
2^{n}}\int_{t^{k}}^{t^{k+1}}\Psi_p (t,\,h^{n,\ast })\,dt  \label{viimply}
\end{equation}%
In view of \eqref{phimaxdef}, (\ref{psimindef}) and (\ref{fmaxdef}), we choose any constant $\mathcal{M}$ such that 
\begin{equation}
\mathcal{M}~>~{\frac{3M^{max}\phi_{max}'}{\psi _{min}^{\prime }}}  \label{upperbound}
\end{equation}%
We then claim that $h_{p}^{n,\ast}(t)\leq \mathcal{M},\,\forall t\in [t_{0},\,t_{f}],\,\forall p\in \mathcal{P},\,\forall n\geq 1$. Otherwise, assume there exists some $p\in \mathcal{P}$, some $\nu \geq 1$
and some $0\leq i\leq 2^{\nu }$ with 
\begin{equation*}
h^{v,\ast }(t)~\equiv ~\eta ~>~\mathcal{M}\qquad t\in \lbrack
t^{i},\,t^{i+1})
\end{equation*}
By choosing $[t_{0},\,t_{f}]$ large enough, we can assume that $i\geq 1$.
Now consider the interval $[t^{i-1},\,t^{i}]$, and the quantity $%
\displaystyle\sup_{t\in \lbrack t^{i-1},t^{i}]}\Psi_p (t,\,h^{v,\ast })$. By
possibly modifying the value of the function $\Psi (\cdot ,\,h^{v,\ast })$
at one point, we can obtain $t^{\ast }\in \lbrack t^{i-1},\,t^{i}]$ such
that 
\begin{equation*}
\Psi_p (t^{\ast },\,h^{v,\ast })~=~\sup_{t\in \lbrack t^{i-1},t^{i}]}\Psi_p
(t,\,h^{v,\ast })
\end{equation*}%
Now let $\tau _{p}(t,\,h^{v,\ast })=t+D_{p}(t,\,h^{v,\ast })$ be the path arrival time function. According to FIFO and assumption {\bf A2}, we
deduce that for any $t\in \lbrack t^{i},\,t^{i+1}]$, 
\begin{equation}
\tau _{p}(t,\,h^{v,\ast })-\tau _{p}(t^{\ast },\,h^{v,\ast })~\geq ~{\frac{%
(t-t^{i})\,\eta }{M^{max}}}  \label{boundedeqn1}
\end{equation}%
this, together with (\ref{psimindef}) implies that 
\begin{equation}
\psi _{ij}\big(\tau _{p}(t,\,h^{v,\ast })\big)-\psi _{ij}\big(\tau
_{p}(t^*,\,h^{v,\ast })\big)~\geq ~\psi _{min}^{\prime }\big(\tau
_{p}(t,\,h^{v,\ast })-\tau _{p}(t^{\ast },\,h^{v,\ast })\big)~\geq ~\psi
_{min}^{\prime }\cdot {\frac{(t-t^{i})\,\eta }{M^{max}}}  \label{boundedeqn2}
\end{equation}%
Inequality (\ref{boundedeqn2})  implies 
\begin{equation}
\Psi _{p}(t,\,h^{v,\ast })-\Psi _{p}(t^{\ast },\,h^{v,\ast })~\geq ~{\frac{%
\psi _{min}^{\prime }\,\eta }{M^{max}}}(t-t^{i})-\phi_{max}'(t-t^{\ast }) \qquad \forall
~t\in \lbrack t^{i},\,t^{i+1}]  \label{boundedeqn3}
\end{equation}%
Integrating (\ref{boundedeqn3}) with respect to $t$ over interval $[t^{i},\,t^{i+1}]$ and a simple
calculation yield 
\begin{multline}
\int_{t^{i}}^{t^{i+1}}\Psi_p (t,\,h^{v,\ast })\,dt-(t^{i+1}-t^{i})\Psi_p
(t^{\ast },\,h^{v,\ast })\\
~\geq ~{\frac{(t^{i+1}-t^{i})^{2}}{2}}\cdot {\frac{%
\psi _{min}^{\prime }\,\eta }{M^{max}}}+(t^{i+1}-t^{i})\phi_{max}'\cdot
\left( t^{\ast }-{\frac{t^{i}+t^{i+1}}{2}}\right)   \label{boundedeqn4}
\end{multline}
Since $t^{\ast }\in \lbrack t^{i-1},\,t^{i}]$, we have $t^{\ast
}-(t^{i}+t^{i+1})/2\geq -3/2(t^{i+1}-t^{i})$. Inequality (\ref{boundedeqn4}) then implies
\begin{equation}
\int_{t^{i}}^{t^{i+1}}\Psi_p (t,\,h^{v,\ast })\,dt-(t^{i+1}-t^{i})\Psi_p
(t^{\ast },\,h^{v,\ast })~\geq ~{\frac{(t^{i+1}-t^{i})^{2}}{2}}\left( {\frac{%
\psi _{min}^{\prime }\,\eta }{M^{max }}}-3\phi_{max}'\right) ~>~0  \label{boundedeqn5}
\end{equation}%
This yields the following contradiction to (\ref{viimply}) and hence (\ref%
{finite}) 
\begin{equation*}
\int_{t^{i}}^{t^{i+1}}\Psi_p (t,\,h^{v,\ast
})\,dt~>~\int_{t^{i-1}}^{t^{i}}\Psi _{p}(t,\,h^{v,\ast })\,dt
\end{equation*}

\noindent \textbf{Part 4.} By previous steps, we have obtained a uniformly bounded
sequence of vector-valued functions $\big\{h^{n,\ast }\big\}_{n\geq 1}$
satisfying (\ref{finite}). By taking a subsequence if necessary, we can assume the weak
convergence in the Hilbert space $\big(L^2[t_0,\,t_f]\big)^{|\mathcal{P}|}$: 
\begin{equation*}
h^{n,\ast }~\longrightarrow ~h^{\ast }\qquad n~\rightarrow ~\infty 
\end{equation*}%
for some $h^{\ast }\in \Lambda $. We claim that for all $h\in \Lambda $, 
\begin{equation*}
\Big<\Psi \big(\cdot ,\,h^{\ast }\big),\,h(\cdot )-h^*(\cdot )\Big>%
~\geq ~0
\end{equation*}%
Indeed, given any $h\in \Lambda $, there exists piecewise-constant
approximation $\{h_{n}\in \Lambda _{n}\}_{n\geq 1}$, which converges to $h$
both point-wise and in the $\Vert \cdot \Vert _{\mathcal{L}^{2}}$ norm.
According to (\ref{finite}), we have 
\begin{equation}
\Big<\Psi _{p}\big(\cdot ,\,h^{n,\ast}\big),\,h_{n}(\cdot)- h^{n,\ast }(\cdot
)\Big>~\geq ~0  \label{eqn1}
\end{equation}%
Notice that the map $h\mapsto \Psi _{p}(\cdot ,\,h)$ is continuous with
respect to the $\Vert \cdot \Vert _{\mathcal{L}^{2}}$ norm, by the
continuity of the inner product, we pass (\ref{eqn1}) to the limit 
\begin{equation*}
\Big<\Psi _{p}\big(\cdot ,\,h^{\ast }\big),\,h(\cdot )-h^*(\cdot )\Big>%
~\geq ~0
\end{equation*}
\end{proof}

The next theorem establishes the existence of SRDC-DUE with generalized Vickrey model, which is an immediate consequence of Theorem \ref{contthm} and Theorem \ref{VIthm}.

\begin{theorem}\label{gvmthm}
Let assumption {\bf A1} hold, then there exists a dynamic user equilibrium as in Definition \ref{duedef}, with an embedded network loading procedure which employs the generalized Vickrey model.
\end{theorem}

\begin{remark}
Notice that assumption {\bf A2} is automatically satisfied by the Vickrey model and hence the GVM. Thus {\bf A2} is omitted from the statement of Theorem \ref{gvmthm}. 
\end{remark}

As our final result, we present an important characterization of the DUE solution $h^*$. That is,  $h_p^*(\cdot),~p\in\mathcal{P}$ are uniformly bounded by some constant. One should distinguish this statement from the assumption that the path flows are {\it a priori} bounded, which is made in the proof of existence \citep{ZM}. Such an informative result leads to further insights of network congestion games, and can be easily derived using analysis similar to {\bf Part 3} in the proof of Theorem \ref{VIthm}. 

\begin{proposition}\label{boundprop} Let {\bf A1} and {\bf A2} hold. Given a DUE solution $h^*(\cdot)\in\Lambda\subset\big(\mathcal{L}^2[t_0,\,t_f]\big)^{|\mathcal{P}|}$ in the sense of Definition \ref{duedef},  the following inequality must hold.
\begin{equation}\label{pwbound}
h_p^*(t)~\leq~{\phi_{max}'\over \psi'_{min}}\cdot M^{max}\qquad\forall t\in[t_0,\,t_f],\quad\forall p\in\mathcal{P}
\end{equation}
where constants $\phi'_{max},\,\psi'_{min}$ and $M^{max}$ are defined in \eqref{phimaxdef}, \eqref{psimindef} and \eqref{fmaxdef}.
\begin{proof}
Fix an origin-destination pair $(i,\,j)\in\mathcal{W}$. Consider any path $p\in\mathcal{P}_{ij}$, and any two departure times $t_0\leq t_1< t_2\leq t_f$ such that $h^*_p(t_1)>0$, $h^*_p(t_2)>0$. Call $\tau_1\leq \tau_2$ the corresponding arrival times after traversing path $p$. By definition, the travel cost for these two drivers must be the same. Recalling (\ref{reveqn1}) and (\ref{vppsidef}), we have
\begin{equation}\label{equalcost}
\phi_{ij}(t_1)+\psi_{ij}(\tau_1)~=~\phi_{ij}(t_2)+\psi_{ij}(\tau_2)
\end{equation}
By the first-in-first-out principle and assumption {\bf A2},
\begin{equation}\label{fifoa2}
\tau_2-\tau_1~\geq~{1\over M^{max}}\int_{t_1}^{t_2}h^*_p(t)\,dt
\end{equation}
Taken together, (\ref{equalcost}) and (\ref{fifoa2}) yield
\begin{align*}
(t_2-t_1)&~\geq~{1\over \phi_{max}'}\left(\phi_{ij}(t_1)-\phi_{ij}(t_2)\right)~=~{1\over \phi_{max}'}\left(\psi_{ij}(\tau_2)-\psi_{ij}(\tau_1)\right)\\
&~\geq~{\psi'_{min}\over \phi_{max}'}(\tau_2-\tau_1)~\geq~{\psi'_{min}\over \phi_{max}' \, M^{max}}\int_{t_1}^{t_2}h^*_p(t)\,dt
\end{align*}
We thus have
\begin{equation}
{\phi_{max}' \, M^{max}\over \psi'_{min}}~\geq~{1\over t_2-t_1}\int_{t_1}^{t_2}h_p^*(t)\,dt
\end{equation}
Since this is true for arbitrary $t_1<t_2$, the point-wise bound (\ref{pwbound}) must hold.
\end{proof}

\end{proposition}

\section{Conclusion}

We have established the existence of the continuous-time simultaneous
route-and-departure choice DUE for the generalized Vickrey model (GVM),
a generalization of the Vickrey model, and plausible regularity conditions
that are easy to check and rather weak. It is significant that ours is the
first DUE existence result without the {\it a priori} bounding of departure rates
(path flows) and the most general constraint relating path flows to the trip
table. In fact, our method of proof successfully overcomes two major hurdles
that have stymied other researchers:  

\begin{enumerate}
\item the set of feasible flows $\Lambda $ is intrinsically non-compact in the $\mathcal{L}^{2}$-space as well as in the $\mathcal{L}^{1}$-space; and

\item a direct topological argument requires \textit{a priori} bounds for
the path flows, where those bounds do not arise from any behavioral argument
or theory.
\end{enumerate}

\noindent Theorem \ref{VIthm} is a general result that subsumes all SRDC-DUE models regardless of the arc dynamics, flow propagation and arc delay function employed, as long as the effective delay operator is continuous. As follow-up research, the continuity of other network loading models will be identified which, when combined with the general result of Theorem \ref{VIthm}, leads to the existence results of DUE with different network performance models.

\bibliographystyle{model2-names}
\bibliography{<your-bib-database>}



\end{document}